\documentclass[11pt,thmsa]{article}%
\usepackage{amsmath}
\usepackage{amsfonts}
\usepackage{amssymb}
\usepackage{graphicx}%
\newtheorem{theorem}{Theorem}[section]

\newtheorem{problem}[theorem]{Problem}

\newtheorem{condition}[theorem]{Condition}

\newtheorem{corollary}[theorem]{Corollary}

\newtheorem{lemma}[theorem]{Lemma}

\newenvironment{proof}[1][Proof]{\noindent\textbf{#1.} }{\ \rule{0.5em}{0.5em}}

\date{}
\begin{document}
\title{Homogeneous Finsler spaces and the flag-wise positively curved condition\thanks{Supported by NSFC (no. 11271198, 51535008)}}

\author{Ming Xu$^1$, Shaoqiang Deng$^2$\thanks{Corresponding author} \\
\\
$^1$College of Mathematics\\
Tianjin Normal University\\
Tianjin 300387, P. R. China\\
Email: mgmgmgxu@163.com.\\
\\
$^2$School of Mathematical Sciences and LPMC\\
Nankai University\\
Tianjin 300071, P. R. China\\
E-mail: dengsq@nankai.edu.cn}

\maketitle

\begin{abstract}
In this paper, we introduce  the flag-wise positively curved condition for Finsler spaces (the (FP) Condition), which means that in each tangent plane, we can find a flag pole in this plane such that the corresponding  flag has   positive flag curvature. Applying the Killing navigation technique, we find a list of compact coset
spaces admitting non-negatively curved homogeneous Finsler metrics
satisfying the (FP) Condition. Using a crucial technique we  developed previously, we prove that most of these coset spaces cannot be endowed with positively curved  homogeneous Finsler metrics. We also prove that any
Lie group whose Lie algebra is a rank $2$ non-Abelian compact Lie algebra
admits a left invariant Finsler metric satisfying the (FP) Condition.
As  by-products, we find
the first example of non-compact coset space $S^3\times \mathbb{R}$ which admits homogeneous flag-wise positively curved Finsler metrics.
Moreover, we find some non-negatively curved Finsler metrics
on $S^2\times S^3$ and $S^6\times S^7$ which satisfy the (FP) Condition, as well as some flag-wise positively curved Finsler metrics on $S^3\times S^3$, shedding some light on the long standing general Hopf conjecture.


\end{abstract}

\section{Introduction}

Many curvature concepts in Riemannian geometry can be naturally generalized  to Finsler geometry. For
example, flag curvature is the generalization of sectional curvature, which
measures how the space curves along a tangent plane. The study of compact Finsler spaces of positive flag curvature (we call these spaces positively curved for simplicity) and the classification of
positively curved homogeneous Finsler spaces are important problems in Finsler
geometry. Recently  we have made some  big progress in the study of these problems; see \cite{XD05}, \cite{XD15}, \cite{XDHH06} and \cite{XZ16}.

It is well known that the flag curvature $K^F(x,y,\mathbf{P})$ depends not only on the tangent plane $\mathbf{P}\subset T_xM$, but also on
 the nonzero base vector $y\in \mathbf{P}$ (the flag pole).
 This implies that flag curvature is much more {\it localized} than sectional curvature in Riemannian geometry.
This feather leaves us more options in studying  flag curvature. For example, concerning the positively curved property
mentioned above, we can define the following {\it flag-wise positively curved condition}, or {\it (FP) Condition} for simplicity.

\begin{condition} We say that a Finsler space $(M,F)$  satisfies the (FP) Condition,
if for any $x\in M$ and tangent plane $\mathbf{P}\subset T_xM$, there exists a nonzero
vector $y\in \mathbf{P}$ such that the flag curvature $K^F(x,y,\mathbf{P})>0$.
\end{condition}

The flag-wise positively curved condition is equivalent to the positively curved condition in Riemannian
geometry, but they seem  essentially different in Finsler geometry. Nevertheless, up to now no example has been constructed to give an interpretation of this phenomenon.

The first purpose of this work is to provide examples of
non-negatively curved homogeneous Finsler spaces satisfying the (FP) Condition. They are $S^1$-bundles over Hermitian
symmetric spaces of compact type.

\begin{theorem} \label{main-thm-1}
Let $N_i=G_i/(S^1\cdot H_i)$ be a compact irreducible Hermitian symmetric space for $1\leq i\leq k$,
and $M$ an $S^1$-bundle over $N=N_1\times\cdots\times N_k$
which can be presented as $M=G/(T^{k-1}\cdot H)$, where
$G=G_1\times\cdots\times G_k$, $H=H_1\times\cdots\times H_k$, and $T^{k-1}$ is a $(k-1)$-dimensional sub-torus in $(S^1)^k\subset G$
which does not contain any $S^1$-factor in this product.
Then $M$ admits $G$-homogeneous Randers metrics which are non-negatively curved and
satisfy the (FP) Condition.
\end{theorem}

In particular, the case $k=1$ for Theorem \ref{main-thm-1}
corresponds to all the canonical $S^1$-bundles over irreducible
Hermitian symmetric spaces of compact type. In the following
corollary, we list all these examples except the homogeneous
spheres $S^{2n-1}=\mathrm{SU}(n)/\mathrm{SU}(n-1)$ which obviously admit positive flag curvature.

\begin{corollary} \label{cor}
The compact coset spaces
\begin{eqnarray}\label{mainlist}
& &\mathrm{SU}(p+q)/\mathrm{SU}(p)\mathrm{SU}(q)\mbox{ with }p\geq q>1,\nonumber\\
& &\mathrm{SO}(n)/\mathrm{SO}(n-2)\mbox{ with }n>4,\quad
\mathrm{Sp}(n)/\mathrm{SU}(n)\mbox{ with }n>1,
\nonumber\\
& &\mathrm{SO}(2n)/\mathrm{SU}(n) \mbox{ with }n>2,\quad
\mathrm{E}_6/\mathrm{SO}(10)\quad\mbox{and}\quad
\mathrm{E}_7/\mathrm{E}_6
\end{eqnarray}
admit non-negatively curved homogeneous Randers metrics satisfying the (FP) Condition.
\end{corollary}

The metrics indicated by the theorem are constructed from Riemannian normal homogeneous metrics and invariant Killing vector fields through the navigation process. The
method is originated from \cite{HD12}, where Z. Hu and S. Deng classified
homogeneous Randers spaces with positive flag curvature and vanishing S-curvature.


By the classification work \cite{BB76} of L. B\'{e}rard Bergery,
it is not hard to see that the coset spaces $M$'s in Theorem
\ref{main-thm-1}
do not admit positively curved homogeneous Riemannian metrics  \cite{BB76}.
In fact, they do not admit positively curved reversible
homogeneous Finsler metrics either \cite{XD15}\cite{XZ16}.
The second purpose of this work is to prove that the same statement holds for most of them without the reversibility condition. This is based on the following crucial criterion for the existence of positively
curved homogeneous metrics on an odd dimensional compact coset space.

\begin{theorem}\label{mainthm-1}
Let $(G/H,F)$ be an odd dimensional positively curved homogeneous Finsler space
such that $\mathfrak{g}=\mathrm{Lie}(G)$ is a compact Lie algebra
with a bi-invariant inner product
$\langle\cdot,\cdot\rangle_{\mathrm{bi}}$, and an orthogonal
decomposition $\mathfrak{g}=\mathfrak{h}+\mathfrak{m}$. Let $\mathfrak{t}$
be a fundamental Cartan subalgebra of $\mathfrak{g}$, i.e.,
$\mathfrak{t}\cap\mathfrak{h}$ is a Cartan subalgebra of $\mathfrak{h}$.
Then there do not exist a pair of linearly independent roots $\alpha$ and $\beta$ of $\mathfrak{g}$ satisfying the
following conditions:
\begin{description}
\item{\rm (1)} Neither $\alpha$ nor $\beta$ is a root of $\mathfrak{h}$;
\item{\rm (2)} $\pm\alpha$ are the only roots of $\mathfrak{g}$ contained in
$\mathbb{R}\alpha+\mathfrak{t}\cap\mathfrak{m}$;
\item{\rm (3)} $\beta$ is the only root of $\mathfrak{g}$ contained in
$\beta+\mathbb{R}\alpha+\mathfrak{t}\cap\mathfrak{g}$.
\end{description}
\end{theorem}

The importance of Theorem \ref{mainthm-1} for studying the classification
of positively curved homogeneous Finsler spaces has been revealed
in \cite{XD15}. Practically, it helps to exclude
many unwanted candidates from the list of positively curved homogeneous Finsler spaces. With the help of additional reversibility assumption, its proof can be very much simplified
(see Lemma 3.10 and its proof in \cite{XD15}).

Theorem \ref{mainthm-1} provides important and surprising insight for the flag curvature of homogeneous Finsler spaces. Before this paper, the only effective application of the homogeneous
flag curvature formula of L. Huang in \cite{Huang2013}
is through its simplified version in \cite{XDHH06} (see Theorem \ref{homogeneous-flag-curvature-thm} and Theorem \ref{simple-flag-curvature-formula-thm} below), which
calculates the flag curvature $K^F(o,u,u\wedge v)$ with linearly independent
$u,v\in \mathfrak{m}=T_{o}(G/H)$ satisfying
\begin{eqnarray}
[u,v]=0\mbox{ and }\langle u,[u,\mathfrak{m}]\rangle_u^F=0.\label{5000}
\end{eqnarray}
Notice that when the homogeneous metric is not reversible, (\ref{5000}) can
not be easily verified.
But in this paper, we find a totally new method which overcomes this obstacle. It sheds much new light on the classification
project for positively curved homogeneous Finsler spaces, in
the most general sense, because most argument in \cite{XD15} and \cite{XZ16} can be carried out without the reversibility assumption.

We apply Theorem \ref{mainthm-1} to the coset space $M$ in Theorem \ref{main-thm-1}. Notice that when $k>1$, for each
fixed $N$, there exist infinitely many $M$ corresponding to it. The following theorem indicates almost all of them can not be
homogeneously positively curved.

\begin{theorem}\label{cor-main-thm-1}
Let $N_i=G_i/(S^1\cdot H_i)$ be a compact irreducible Hermitian symmetric space for $1\leq i\leq k$,
and $M$ an $S^1$-bundle over $N=N_1\times\cdots\times N_k$
which can be presented as $M=G/(T^{k-1}\cdot H)$, where
$G=G_1\times\cdots\times G_k$, $H=H_1\times\cdots\times H_k$ and $T^{k-1}$ is a $(k-1)$-dimensional sub-torus in $(S^1)^k\subset G$
which does not contain any $S^1$-factor in this product. Then
we have the following:
\begin{description}
\item{\rm (1)} When $k=1$, none of the coset spaces
\begin{eqnarray}\label{mainlist-2}
& &\mathrm{SU}(p+q)/\mathrm{SU}(p)\mathrm{SU}(q), \mbox{ with }p>q\geq 2
\mbox{ or }p=q>3,\nonumber\\
& &\mathrm{Sp}(n)/\mathrm{SU}(n),\mbox{ with }n>4,\nonumber\\
& &\mathrm{SO}(2n)/\mathrm{SU}(n), \mbox{ with }n=5\mbox{ or }n>6,
\nonumber\\
& &\mathrm{E}_6/\mathrm{SO}(10)\mbox{ and }\mathrm{E}_7/\mathrm{E}_6
\end{eqnarray}
admits positively curved homogeneous Finsler metrics.
\item{\rm (2)} When $k=2$ or $3$, for each $N$, there may exist only finitely many coset spaces $M$'s corresponding to it, which admit positively curved homogeneous Finsler metrics.
\item{\rm (3)} When $k>3$, none of the coset spaces $M$'s admit positively curved homogeneous Finsler metrics.
\end{description}
\end{theorem}

To summarize, Theorem \ref{main-thm-1} and Theorem \ref{cor-main-thm-1} together provide abundant examples of compact coset spaces
which admit non-negatively
and flag-wise positively curved homogeneous Finsler metrics, but no positively curved homogeneous Finsler metrics.
This raises the problem to classify the normal homogeneous Finsler spaces which satisfy the (FP) Condition. We will take up this problem in the future.  Here we will only pose some problems
related to the above  examples, and  conjecture that the answer to these problem will be  positive.

Concerning the (FP) Condition alone, without the non-negatively curved condition, we prove

\begin{theorem}\label{main-thm-2}
Let $G$ be a Lie group whose  Lie algebra $\mathfrak{g}$ is
compact and non-Abelian with $\mathrm{rk}\mathfrak{g}=2$. Then $G$ admits
left invariant Finsler metrics satisfying the (FP) Condition.
\end{theorem}

According to \cite{DH2013}, the group in Theorem \ref{main-thm-2} does not admit any left invariant positively curved Finsler metrics. Note that the Lie groups in this theorem can even be non-compact,
e.g.,  $\mathrm{SU}(2)\times\mathbb{R}=S^3\times\mathbb{R}$, which does not
admit any positively curved homogeneous Finsler metric by the Bonnet-Myers Theorem
in Finsler geometry \cite{BCS00}.

There are some interesting by-products of the main results of this paper.
Applying Theorem \ref{main-thm-1} and Corollary \ref{cor} to $\mathrm{SO}(4)/\mathrm{SO}(2)$
(which is diffeomorphic to $S^2\times S^3$) and
$\mathrm{SO}(8)/\mathrm{SO}(6)$ (which is diffeomorphic to $S^6\times S^7$) respectively,
and applying Theorem \ref{main-thm-2} to $\mathrm{SU}(2)\times\mathrm{SU}(2)$
(which is diffeomorphic to $S^3\times S^3$),
we obtain
\begin{corollary}\label{by-product-corollary}
\begin{enumerate}
\item[(1)] Both $S^2\times S^3$ and $S^6\times S^7$ admit non-negatively curved Finsler metrics satisfying the (FP) Condition.
\item[(2)] $S^3\times S^3$ admits Finsler metrics satisfying the (FP) Condition.
\end{enumerate}
\end{corollary}

Note that the existence of a positively curved Riemanian metric on a product manifold, which is a part of the general Hopf conjecture,
is a very hard long standing open problem. The interest of the above corollary lies in the fact
that at least there exist  non-negatively curved Finsler metrics satisfying
the (FP) Condition on certain product manifolds, and it sheds some light on such an involved and significant problem.

Finally, we have two remarks on some recent
progress inspired by this paper.

In \cite{X2017}, the first author pointed out that the (FP)
condition alone is relatively weak. Perturb a non-negatively
curved Finsler metric, one has a big chance to get
a flag-wise positively curved metric. Using the Killing navigation process and the gluing technique for proving
Theorem \ref{main-thm-2}, he  gave a positive answer for Problem \ref{question-3}, and provided abundant examples of
non-homogeneous flag-wise positively curved Finsler space.

In \cite{XZ17}, the first author and L. Zhang pointed out that
the metrics we have constructed on these examples are normal homogeneous or $\delta$-homogeneous. For normal homogeneous or $\delta$-homogeneous Finsler spaces satisfying the (FP) Condition, Problem \ref{question-1} and Problem \ref{question-2} were positively justified. For even dimensional
normal homogeneous Finsler spaces satisfying the (FP) Condition, a complete classification were given.

This paper is organized as following. In Section 2, we review some basic concepts in Finsler geometry which will be needed for this work.  We also introduce the Killing navigation process and a key flag curvature formula as well as the homogeneous flag curvature formulas. In Section 3,
we use the Killing navigation method to construct flag-wise positively and non-negatively curved homogeneous Randers metrics, and prove Theorem \ref{main-thm-1}.
In Section 4, we prove Theorem \ref{main-thm-2} and raise more questions concerning the (FP) Condition.
In Section 5, we prove Theorem \ref{mainthm-1}. In Section 6,
we prove Theorem \ref{cor-main-thm-1}.

\medskip
\noindent {\bf Acknowledgements.} We would like to thank W. Ziller and L. Huang
for helpful comments and discussions. This work is dedicated to the
80th birthday of our most respectable friend, teacher of mathematics and life, Professor Joseph A. Wolf.
\section{Preliminaries}
\subsection{Finsler metric and Minkowski norm}

A Finsler metric on a smooth manifold $M$ is a continuous function
$F:TM\rightarrow [0,+\infty)$ satisfying the following conditions:
\begin{description}
\item{\rm (1)} $F$ is a smooth positive function on the
slit tangent bundle $TM\backslash 0$;
\item{\rm (2)} $F(x,\lambda y)=\lambda F(x,y)$ for any $x\in M$, $y\in T_xM$, and $\lambda\geq0$;
\item{\rm (3)} On any
{\it standard local coordinates} $x=(x^i)$ and $y=y^i\partial_{x^i}$ for $TM$,    the Hessian matrix $$(g_{ij}(x,y))=\frac12[F^2(x,y)]_{y^iy^j}$$ is positive definite for any nonzero $y\in T_xM$, i.e., it defines an inner product (sometimes denoted as  $g_y^F$ )
$$\langle u,v\rangle_y^F=\frac12\frac{\partial^2}{\partial s
\partial t}F^2(y+su+tv)|_{s=t=0},\quad u,v\in T_xM,$$
on $T_xM$.
\end{description}
We call $(M,F)$ a Finsler space or a Finsler manifold. The restriction of
the Finsler metric to a tangent space is called a Minkowski norm. Minkowski norm can also be defined on any real vector space by similar conditions as
(1)-(3), see \cite{BCS00} and \cite{CS05}. A Finsler metric (or a Minkowski norm) $F$ is called {\it reversible} if $F(x,y)=F(x,-y)$ for all $x\in M$
and $y\in T_xM$ (or $F(y)=F(-y)$ for all vector $y$, respectively).

Riemannian metrics are the special Finsler metrics such that for any standard
local coordinates the Hessian matrix $(g_{ij}(x,y))$ is independent of $y$.
In this case, $g_{ij}(x,y)dx^idx^j$ is a well defined global section of
$\mathrm{Sym}^2(T^*M)$, which is often referred to as the Riemannian metric.
The most important and simple non-Riemannian Finsler metrics are Randers metrics \cite{Ra41}, which are of the form $F=\alpha+\beta$,  where $\alpha$ is a Riemannian
metric, and $\beta$ is a 1-form. The condition for (the Hessian of) $F$ to be positive definite is that
 the $\alpha$-norm of $\beta$  is smaller than $1$ at each point, see \cite{BCS00}.
 $(\alpha,\beta)$-metrics are generalizations of
Randers metrics. They are of the form $F=\alpha\phi(\beta/\alpha)$, where
$\alpha$ and $\beta$ are the same as in Randers metrics, and $\phi$ is
a smooth function. The condition for $F$ to be positive definite can be found in \cite{CS05}. Recently, there are many research works on $(\alpha,\beta)$-spaces, see for example \cite{BCS07}.

The differences between Riemannian metrics and non-Riemannian Finsler metrics
can be detected by the {\it Cartan tensor} $C^F_y(u,v,w)$, where $y$, $u$, $v$ and $w$ are tangent vectors in  $T_xM$,
and  $y\neq 0$. The Cartan tensor is defined by
$$C^F_y(u,v,w)=\frac14\frac{\partial^3}{\partial r\partial s\partial t}
F^2(x,y+ru+sv+tw)|_{r=s=t=0},$$
or equivalently,
$$C^F_y(u,v,w)=\frac12\frac{d}{dt}\langle u,v\rangle_{y+tw}^F|_{t=0}.$$
A Finsler metric is Riemannian if and only if its Cartan tensor vanishes everywhere.
The Cartan tensor can also be defined for a Minkowski norm which detects whether  the norm is induced by  an inner product.

\subsection{Riemann curvature and flag curvature}

The Riemann curvature for a Finsler space $(M,F)$ can be either defined from the structure equation for Chern connection, or by the Jacobi equation for
a variation of geodesics of constant speeds \cite{CS05}. For any
standard local coordinate system, we can
present it as a linear operator $R^F_y=R^i_k(y)\partial_{x^i}\otimes dx^k:T_xM\rightarrow T_xM$,
where
$$R^i_k(y)=2\partial_{x^k}\mathbf{G}^i
-y^j\partial^2_{x^jy^k}\mathbf{G}^i+2\mathbf{G}^j\partial^2_{y^jy^k}
\mathbf{G}^i
-\partial_{y^j}\mathbf{G}^i\partial_{y^k}\mathbf{G}^j,$$
and
$$\mathbf{G}=y^i\partial_{x^i}-2\mathbf{G}^i\partial_{y^i}$$
is the coefficients of the geodesic spray field $\mathbf{G}=y^i\partial_{x^i}-2\mathbf{G}^i\partial_{y^i}$ on
$TM\backslash 0$.

Flag curvature is a natural generalization of sectional curvature in
Riemannian geometry. Now let us recall the definition in \cite{BCS00}. Given a  nonzero vector $y\in T_xM$ (the flag pole),
and a tangent plane $\mathbf{P}\subset T_xM$ spanned by $y$ and $w$ (the flag), the flag curvature is defined as
\begin{equation}\label{5999}
K^F(x,y,\mathbf{P})=K^F(x,y,y\wedge w)=\frac{\langle R^F_y w,w\rangle^F_y}{\langle w,w\rangle^F_y\langle y,y\rangle^F_y-[{\langle y,w\rangle^F_y}]^2}.
\end{equation}
Notice that the flag curvature  depends only  on $y$ and $\mathbf{P}$ rather than $w$. When $F$ is a Riemannian metric, it is just the sectional curvature and irrelevant to the choice
of the flag pole.

\subsection{Navigation process and the flag curvature formula for Killing navigation}
The navigation process is an important technique in studying Randers spaces and flag curvature \cite{BRS04}. Let $V$ be a vector field on the Finsler space $(M,F)$ with
$F(V(x))<1$, for any $x\in M$. Given  $y\in T_xM$, denote
$\tilde{y}=y+F(x,y)V(x)$. Then $\tilde{F}(x,\tilde{y})=F(x,y)$ defines
a new Finsler metric on $M$. We call it the metric defined by the navigation
datum $(F,V)$. When $V$ is a Killing vector field of $(M,F)$, i.e.,
$L_VF=0$, we call this process a Killing navigation, and $(F,V)$ a Killing navigation datum. The following flag curvature formula for Killing navigation is crucial in this work.

\begin{theorem}\label{killingnavigationthm}
Let $\tilde{F}$ be the metric defined by the Killing navigation datum
$(F,V)$ on the smooth manifold $M$ with $\dim M>1$. Then for any $x\in M$, and any nonzero vectors $w$ and $y$ in $T_xM$ such that $\langle w,y\rangle^F_y=0$, we have $K^F(x,y,y\wedge w)=K^{\tilde{F}}(x,\tilde{y}, \tilde{y}\wedge w)$.
\end{theorem}

The proof can be found in \cite{HM07} or \cite{HM15}, where  some more general situations are also considered.

\subsection{Homogeneous flag curvature formulas}
Using local coordinates to study the flag curvature in homogenous Finsler
geometry is not very convenient. In \cite{Huang2013}, L. Huang provided a
homogeneous flag curvature formula which is expressed by some
 algebraic data of the homogeneous Finsler space.

To introduce his formula, we first define some notations.

Let $(G/H,F)$ be a homogeneous Finsler space,  where $H$ is the compact isotropy subgroup at $o=eH$,
$\mathrm{Lie}(G)=\mathfrak{g}$ and
$\mathrm{Lie}(H)=\mathfrak{h}$. Then we have
an $\mathrm{Ad}(H)$-invariant decomposition $\mathfrak{g}=\mathfrak{h}+\mathfrak{m}$. Denote the projection to the $\mathfrak{h}$-factor (or $\mathfrak{m}$-factor) in this decomposition as $\mathrm{pr}_\mathfrak{h}$ (or $\mathrm{pr}_\mathfrak{m}$), and define
$[\cdot,\cdot]_\mathfrak{h}=\mathrm{pr}_\mathfrak{h}\circ[\cdot,\cdot]$
and $[\cdot,\cdot]_\mathfrak{m}=\mathrm{pr}_\mathfrak{m}\circ[\cdot,\cdot]$,
respectively. Notice that $\mathfrak{m}$ can be canonically identified with
the tangent space $T_o(G/H)$.

For any $u\in\mathfrak{m}\backslash\{0\}$, the {\it spray vector field}
$\eta(u)$ is defined by
\begin{equation*}
	\langle\eta(u), w\rangle_u^F=\langle u,[w,u]_{\mathfrak{m}}\rangle_u^F, \quad \forall w\in\mathfrak{m},
\end{equation*}
and the {\it connection operator} $N(u,\cdot)$ is a linear operator
on $\mathfrak{m}$ determined by
\begin{equation}
	\begin{aligned}
	2\langle N(u,w_1),w_2\rangle_u^F=\langle [w_2,w_1]_{\mathfrak{m}}, u\rangle_u^F & + \langle [w_2,u]_{\mathfrak{m}}, w_1\rangle_u^F+\langle [w_1,u]_{\mathfrak{m}},w_2\rangle_u^F\nonumber\\
	& - 2C^F_u(w_1,w_2,\eta(u)),\quad \forall w_1,w_2\in\mathfrak{m}.\label{define-connection-operator}
	\end{aligned}	
\end{equation}

L. Huang has proven the following theorem.
\begin{theorem}\label{homogeneous-flag-curvature-thm}
Let $(G/H,F)$ be a homogeneous Finsler space, and
$\mathfrak{g}=\mathfrak{h}+\mathfrak{m}$ an $\mathrm{Ad}(H)$-invariant
decomposition, where $\mathfrak{g}=\mathrm{Lie}(G)$,
$\mathfrak{h}=\mathrm{Lie}(H)$, and $\mathfrak{m}=T_o(G/H)$.
Then for any nonzero vector
$u\in T_o(G/H)$, the Riemann curvature
$R^F_u: T_o(G/H)\rightarrow T_o(G/H)$ satisfies
\begin{equation}\label{6000}
	\langle R^F_u v,v\rangle_u^F
	=\langle [[v,u]_{\mathfrak{h}},v],u\rangle_u^F
	+\langle \tilde{R}(u)w,w\rangle_u^F,\quad \forall w\in\mathfrak{m},
\end{equation}
where the linear operator
$\tilde{R}(u):\mathfrak{m}\rightarrow\mathfrak{m}$
is given by
$$\tilde{R}(u)v=D_{\eta(u)}N(u,v)-N(u,N(u,v))+
N(u,[u,v]_\mathfrak{m})-[u,N(u,v)]_\mathfrak{m},$$
and $D_{\eta(u)}N(u,v)$ is the derivative of $N(\cdot,v)$ at $u\in\mathfrak{m}\backslash\{0\}$ in the direction of $\eta(u)$.
\end{theorem}
This  provides an explicit presentation for the nominator in (\ref{5999}), which
is the only crucial term for studying the flag curvature. So we call (\ref{6000}) the {\it homogeneous flag curvature formula}.

In particular, when $\eta(u)=0$
(i.e. $\langle u,[\mathfrak{m},u]_\mathfrak{m}\rangle_u^F=0$), and
$\mathrm{span}\{u,v\}$ is a two dimensional commutative subalgebra of $\mathfrak{g}$,
we get a  very simple and useful  homogeneous flag curvature.

\begin{theorem} \label{simple-flag-curvature-formula-thm}
Let $(G/H,F)$ be a connected homogeneous Finsler space, and $\mathfrak{g}=\mathfrak{h}+\mathfrak{m}$ be an $\mathrm{Ad}(H)$-invariant
decomposition for $G/H$. Then for any linearly independent commutative pair
$u$ and $v$ in $\mathfrak{m}$ satisfying
$
\langle[u,\mathfrak{m}],u\rangle^F_u=0
$,
we have
\begin{equation*}
K^F(o,u,u\wedge v)=\frac{\langle U(u,v),U(u,v)\rangle_u^F}
{\langle u,u\rangle_u^F \langle v,v\rangle_u^F-
[\langle u,v\rangle_u^F]^2},
\end{equation*}
where $U:\mathfrak{m}\times \mathfrak{m}\rightarrow\mathfrak{m}$ is defined by
\begin{equation*}
\langle U(u,v),w\rangle_u^F=\frac{1}{2}(\langle[w,u]_\mathfrak{m},v\rangle_u^F
+\langle[w,v]_\mathfrak{m},u\rangle_u^F), \quad \forall w\in\mathfrak{m}.
\end{equation*}
\end{theorem}

Theorem \ref{simple-flag-curvature-formula-thm} can be proven either by using
Theorem \ref{homogeneous-flag-curvature-thm} or directly using the Finslerian
submersion technique. See \cite{XDHH06} for the details.
\section{Non-negatively curved homogeneous Finsler spaces satisfying the  (FP) Condition}

In this section, we will construct some examples of compact non-negatively curved homogeneous Finsler spaces satisfying the (FP) Condition. Our main tool is
the following lemma.

\begin{lemma}\label{proposition-main-tool}
Let $G$ be a compact Lie group, $H$ a closed connected subgroup, $\mathrm{Lie}(G)=\mathfrak{g}$ and $\mathrm{Lie}(H)=\mathfrak{h}$.
 Fix a bi-invariant inner product on $\mathfrak{g}=\mathrm{Lie}(G)$, denoted as $\langle\cdot,\cdot\rangle_{\mathrm{bi}}$, and  an orthogonal  decomposition  $\mathfrak{g}=\mathfrak{h}+\mathfrak{m}$.
Assume that there exists a vector $v\in\mathfrak{m}$ satisfying the following conditions:
\begin{description}\label{proposition-3-1}
\item {\rm (1)} There exists a one-dimensional subspace
$\mathfrak{m}_0$ in $\mathfrak{m}$ such that the action of $\mathrm{Ad}(H)$ on $\mathfrak{m}_0$ is  trivial;
\item {\rm (2)} Any 2-dimensional commutative subalgebra $\mathfrak{t}'\subset\mathfrak{m}$ satisfies $\langle\mathfrak{m}_0,\mathfrak{t}'\rangle_{\mathrm{bi}}=0$.
\end{description}
Then $G/H$ admits non-negatively curved homogeneous Randers metrics
satisfying the (FP) Condition
\end{lemma}

\begin{proof}
First we endow $G/H$ with a Riemannian normal homogeneous metric $F$ defined by a bi-invariant inner product $\langle\cdot,\cdot\rangle_{\mathrm{bi}}$ on $\mathfrak{m}$.
By (1) of the lemma, we can find a vector $v$ satisfying $[v,\mathfrak{h}]=0$ and
$\langle v,v\rangle_{\mathrm{bi}}<1$. The left invariant vector field $v\in\mathfrak{m}\subset\mathfrak{g}$ induced on $G$ also defines a $G$-invariant
vector field $V$ on $G/H$. Furthermore, since for any $u\in\mathfrak{m}$, we have
$$\langle [v,u]_\mathfrak{m},u\rangle_{\mathrm{bi}}=\langle [v,u],u\rangle_{\mathrm{bi}}=0,$$
and
$$\langle [v,u]_\mathfrak{m},v\rangle_\mathrm{bi}=\langle
[v,u],v\rangle_{\mathrm{bi}}=0,$$
$V$ is a Killing vector field
of $(G/H,F)$.
 Therefore $(F,V)$ is a Killing navigation datum and it defines a Finsler metric $\tilde{F}$ on $G/H$ by the navigation process.
Since both $F$ and $V$ are $G$-invariant, so is $\tilde{F}$.

Note that the Riemannian normal homogeneous space $(G/H,F)$ is non-negatively curved. So by  Theorem \ref{killingnavigationthm}, $(G/H,\tilde{F})$ is also non-negatively
curved. To show that $\tilde{F}$ satisfies the (FP) Condition, we only need to prove the assertion at $o=eH$. The tangent space $T_{o}(G/H)$ can be identified with
$\mathfrak{m}$. Now let $\mathbf{P}\subset\mathfrak{m}$ be a tangent plane at $o$.  Then there are only the following two cases.

\textbf{Case 1}\quad   $\mathbf{P}$ is a commutative subalgebra. By (2) of
 the lemma, $\langle\mathbf{P},v\rangle_{\mathrm{bi}}=0$. Fixing an
 orthogonal basis $\{v_1,v_2\}$ of $\mathbf{P}$, we can find a nonzero vector $v'_1$ such that
$$\tilde{v'_1}=v'_1+F(v'_1)v=v_1.$$
Since
$\langle v'_1,v_2\rangle^F_{v'_1}=\langle v'_1,v_2\rangle_{\mathrm{bi}}=0$,
 by Theorem \ref{killingnavigationthm}, we have
$$K^{\tilde{F}}(o,v_1,v_1\wedge v_2)=K^F(o,v'_1,v'_1\wedge v_2).$$
On the other hand, by (2) of the lemma, we have $[v,v_2]\neq 0$. Therefore
$$[v'_1,v_2]=-F(v'_1)[v,v_2]\neq 0.$$
Then by the curvature formula of the Riemannian normal
homogeneous spaces, we have $$K^{\tilde{F}}(o,v_1,\mathbf{P})=K^F(o,v'_1,v'_1\wedge v_2)>0.$$

\textbf{Case 2}\quad  $\mathbf{P}$ is not a commutative subalgebra in $\mathfrak{m}$. Fix a bi-invariant  orthogonal basis $\{v_1,v_2\}$ of $\mathbf{P}$ such that
$\langle v,v_2\rangle_{\mathrm{bi}}=0$. Select the nonzero vectors $v'_1$ and $v''_1$
such that
$$\tilde{v'_1}=v'_1+F(v'_1)v=v_1\mbox{ and }
\tilde{v''_1}=v''_1+F(v''_1)v=-v_1.$$ If $[v'_1,v_2]\neq 0$ or $[v''_1,v_2]\neq 0$, then by a similar argument as above, we can get
$K^{\tilde{F}}(o,v_1,\mathbf{P})\neq 0$ or
$K^{\tilde{F}}(o,-v_1,\mathbf{P})\neq 0$. If $[v'_1,v_2]=[v''_1,v_2]=0$, we then have $[v,v_2]=[v_1,v_2]=0$, which is a contradiction.

Therefore,  the homogeneous metric $\tilde{F}$ on $G/H$ is non-negatively curved and
satisfies the (FP) Condition. This completes the proof of the lemma.
\end{proof}

The assumptions in Lemma \ref{proposition-3-1} in fact imply that
$\mathrm{rk}\mathfrak{g}=\mathrm{rk}\mathfrak{h}+1$. Moreover, we assert that $H$ is a regular subgroup of
$G$. In fact,  any Cartan subalgebra of $\mathfrak{h}$ can be extended to a fundamental Cartan subalgebra $\mathfrak{t}$ of $\mathfrak{g}$ by adding $\mathfrak{m}_0$.
For simplicity, we call such $\mathfrak{t}$ a fundamental Cartan subalgebra,
because $\mathfrak{t}\cap\mathfrak{h}$ is  a Cartan subalgebra of $\mathfrak{h}$. By (1) of Lemma \ref{proposition-3-1}, we have $[\mathfrak{m}_0,\mathfrak{h}]=0$. Hence
any root plane of $\mathfrak{h}$ with respect to $\mathfrak{t}\cap\mathfrak{h}$ is also a root plane of $\mathfrak{g}$ with
respect to $\mathfrak{t}$. Thus $H$ is a regular subgroup of $G$.

If we further bi-invariant orthogonally decompose $\mathfrak{m}$ as
$\mathfrak{m}=\mathfrak{m}_0+\mathfrak{m}'$, then
$\mathfrak{m}_1$ is the sum of some root planes with respect to the Cartan subalgebra $\mathfrak{t}$ mentioned above.

Now we are ready to prove Theorem \ref{main-thm-1}.

\medskip
\noindent\textbf{Proof of Theorem \ref{main-thm-1}.}
For each irreducible Hermitian symmetric space
$N_i=G_i/K_i=G_i/(S^1\cdot H_i)$, we denote
$\mathfrak{g}_i=\mathfrak{k}_i+\mathfrak{p}_i$
with $\mathfrak{k}_i=\mathfrak{h}_i+\mathbb{R}v_i$ its bi-invariant orthogonal decomposition.
Then for $M=G/H'$ with $H'=T^{k-1}\cdot H$
and $H=H_1\times\cdots\times H_k$,
we have
the bi-invariant orthogonal decomposition
$\mathfrak{g}=\mathfrak{h}'+\mathfrak{m}$, where
$\mathfrak{m}=\mathbb{R}v+\mathfrak{p}$, $\mathfrak{p}=\sum_{i=1}^k \mathfrak{p}_i$, and
$\mathfrak{h}'=\mathfrak{h}+\mathrm{Lie}(T^{k-1})$ with
$\mathfrak{h}=\bigoplus_{i=1}^k\mathfrak{h}_i$ is the bi-invariant orthogonal complement
of $\mathfrak{m}_0=\mathbb{R}v$ in
$\mathfrak{k}=\bigoplus_{i=1}^k\mathfrak{k}_i$.
Notice that $v=\sum_{i=1}^k c_iv_i$ for some coefficients $c_i$'s.
Our assumption on $M$ implies all $c_i$'s are nonzero.
Fix a Cartan subalgebra $\mathfrak{t}$ of $\mathfrak{k}$, then it contains all $v_i$'s, and $\mathfrak{t}\cap\mathfrak{h}'$ is
a Cartan subalgebra of $\mathfrak{h}'$, i.e. $\mathfrak{t}$ is a fundamental
Cartan subalgebra of $\mathfrak{g}$ for the decomposition
$\mathfrak{g}=\mathfrak{h}'+\mathfrak{m}$. The vector subspace
$\mathfrak{t}\cap\mathfrak{m}$ coincides with
$\mathfrak{m}_0=\mathbb{R}v$. For any $i$, and any root plane $\mathfrak{g}_{\pm\alpha}\subset\mathfrak{p}_i$,
we have $\langle v,\alpha\rangle_{\mathrm{bi}}=c_i\langle v_i,\alpha\rangle_{\mathrm{bi}}\neq 0$. Therefore for any nonzero $v'\in\mathfrak{p}$, the Lie bracket
$[v,v']$ is a nonzero vector in $\mathfrak{p}$.

Now we apply Lemma \ref{proposition-3-1} to prove the theorem.
Condition (1) of Lemma \ref{proposition-3-1} is naturally satisfied. Now we consider Condition (2) of Lemma \ref{proposition-3-1}. Given any 2-dimensional commutative
subalgebra $\mathfrak{t}'$ of $\mathfrak{g}$ contained in
$\mathfrak{m}$, there exists a bi-invariant orthogonal basis
of it, consisting of $u_1=av+u'_1$ and $u_2$, where $a\in\mathbb{R}$ and $u'_1,u_2\in\mathfrak{p}$. The commutative
condition implies $[u_1,u_2]=a[v,u_2]+[u'_1,u_2]=0$. By the symmetric condition, we have $[u'_1,u_2]\in\mathfrak{k}'=\mathfrak{h}'+\mathfrak{m}_0$. Moreover, we have seen $[v,u_2]$ is a nonzero vector in $\mathfrak{p}$. Therefore
we must have $a=0$, i.e. $\langle v,\mathfrak{t}'\rangle_{\mathrm{bi}}=0$. This implies that
Condition (2) of Lemma \ref{proposition-3-1} is also satisfied. So $G/H$ admits a non-negatively curved homogeneous Finsler
metric satisfying the (FP) Condition, which ends the proof of
Theorem \ref{main-thm-1}.

\medskip
In the case $k=1$ for Theorem \ref{main-thm-1}, $N$ is
an irreducible Hermitian symmetric space of compact type,
and $M$ is the canonical $S^1$-bundle over $N$ which
satisfies the assumption in Theorem \ref{main-thm-1}
automatically.
All compact irreducible Hermitian symmetric pairs $(\mathfrak{g},\mathfrak{k})$ can be listed as the following:
\begin{eqnarray*}
& &(A_{p+q-1},A_{p-1}\oplus A_{q-1}\oplus\mathbb{R}),
(B_n,B_{n-1}\oplus\mathbb{R}), (C_n,A_n\oplus\mathbb{R}),
(D_n,D_{n-1}\oplus\mathbb{R}), \\
& &(E_6, D_5\oplus\mathbb{R})\mbox{ and }
(E_7,E_6\oplus\mathbb{R}).
\end{eqnarray*}
So we get the list (\ref{mainlist}) in Corollary \ref{cor} accordingly.

\section{Left invariant metrics on quasi-compact Lie groups satisfying the (FP) Condition}

All the examples of non-negatively curved homogeneous Finsler spaces satisfying the (FP) Condition provided by Theorem \ref{main-thm-1} are compact. Furthermore,
they all satisfy the rank inequality $\mathrm{rk}\mathfrak{g}\leq \mathrm{rk}\mathfrak{h}+1$ which is satisfied by all the positively curved homogeneous spaces. These observations lead to the following problems.
\begin{problem} \label{question-1}
  Is a non-negatively curved homogeneous Finsler space satisfying the (FP) Condition   compact?
\end{problem}
\begin{problem} \label{question-2}
Does any non-negatively curved and flag-wise positively curved homogeneous Finsler spaces
$(G/H,F)$ with compact $G$ satisfy the condition $\mathrm{rk}\mathfrak{g}\leq\mathrm{rk}\mathfrak{h}+1$ for their Lie algebras?
\end{problem}

In Riemannian geometry, the (FP) Condition   is equivalent to the positively curved condition, hence the answers to both the above problems are  positive.
In Finsler geometry, however, the non-negatively curved condition seems to be more subtle, since Theorem \ref{main-thm-2} provides counter examples for
both problems with the non-negatively curved condition dropped.

\medskip
\noindent\textbf{Proof of Theorem \ref{main-thm-2}.} We start with the bi-invariant Riemannian metric $F$ on $G$, which is determined by the bi-invariant inner product
$\langle\cdot,\cdot\rangle_{\mathrm{bi}}$. Let $\mathfrak{t}_1$ be a Cartan
subalgebra in $\mathfrak{g}$ and $v_1$ a generic vector in $\mathfrak{t}_1$, i.e.,
$\mathfrak{t}_1=\mathfrak{c}_{\mathfrak{g}}(v_1)$. Then $v_1$ defines a left invariant vector field $V_1$ on $G$, which is also a Killing vector field of
constant length for
$(G,F)$. For any fixed sufficiently small $\epsilon>0$, the navigation datum $(F,\epsilon V_1)$ defines a Finsler metric $\tilde{F}_{1,\epsilon}$. Since  both $F$ and $V_1$ are left invariant, so is $\tilde{F}_{1,\epsilon}$.
As we have argued before, $(G,\tilde{F}_{1,\epsilon})$ is non-negatively curved.
The following lemma indicates that $\mathfrak{t}_1$ is the only $2$-dimensional subspace $\mathbf{P}\subset\mathfrak{g}=T_eG$ for which the (FP) Condition may not be satisfied for $(G,\tilde{F}_{1,\epsilon})$.

\begin{lemma} \label{lemma-007}
Keep all the above assumptions and notations, and fix the sufficiently small $\epsilon>0$. If the $2$-dimensional subspace
$\mathbf{P}\subset\mathfrak{g}$ satisfies $K^{\tilde{F}_{1,\epsilon}}(e,y,\mathbf{P})\leq 0$ for any nonzero $y\in\mathbf{P}$, then $\mathbf{P}=\mathfrak{t}_1$.
\end{lemma}
\begin{proof}
 Given any $\mathbf{P}\subset\mathfrak{g}$ as in the lemma, we can find a nonzero vector $w_2\in\mathbf{P}$ with $\langle w_2,v_1\rangle_{\mathrm{bi}}=0$. Then there exists
a nonzero vector $w_1\in\mathbf{P}$ such that $\langle w_1,w_2\rangle_{\mathrm{bi}}=0$. One can also find  some suitable positive numbers $a$ and $b$, such that $w'_1=w_1-av_1$
and $w'_2=-w_1-bv_1$ satisfying the conditions  $$\tilde{w'_1}=w'_1+\epsilon F(w'_1) v_1=w_1\mbox{ and }
\tilde{w'_2}=w'_2+\epsilon F(w'_2)v_1=-w_1.$$ Moreover, we also have
$$\langle w'_1,w_2\rangle_{\mathrm{bi}}=\langle w'_2,w_2\rangle_{\mathrm{bi}}=0.$$
 By Theorem \ref{killingnavigationthm}, we have
\begin{eqnarray}
K^{F}(e,w'_1\wedge w_2)&=&K^{\tilde{F}_{1,\epsilon}}(e,w_1,\mathbf{P})\leq 0,\label{0000}
\end{eqnarray}
 and
 \begin{eqnarray}
K^F(e,w'_2\wedge w_2)&=&K^{\tilde{F}_{1,\epsilon}}(e,-w_1,\mathbf{P})\leq 0.\label{0001}
\end{eqnarray}
 Since $(G,F)$ is non-negatively curved, the  equality holds for both (\ref{0000}) and (\ref{0001}), that is,  we have
$$
[w'_1,w_2]=[w_1,w_2]-a[v_1,w_2]=0,$$
 and
$$[w'_2,w_2]=-[w_1,w_2]-b[v_1,w_2]=0,$$
from which we conclude that $[w_1,w_2]=[v,w_2]=0$.
Sice  $v_1$ is a generic vector in $\mathfrak{t}_1$, we have $w_2\in\mathfrak{t}_1$.
Now if we change the flag pole to another generic $w_3=w_1+cw_2\in\mathbf{P}$,
then there is a nonzero number $d$ such that the vector $w_4=w_2+dw_1$ satisfies the condition
$\langle w_3,w_4\rangle^{\tilde{F}_{1,\epsilon}}_{w_3}=0$. Let $w'_3$ be the nonzero vector
such that $\tilde{w'_3}=w'_3+\epsilon F(w'_3)v=w_3$. Then
$\langle w'_3,w_4\rangle_{w'_3}^F=\rangle_{\mathrm{bi}}=0$, and by Theorem \ref{killingnavigationthm}, we have
$$K^F(e,w'_3\wedge w_4)=K^{\tilde{F}_{1,\epsilon}}(e,w_3,w_3\wedge w_4)\leq 0.$$
So we have $K^F(e,w'_3\wedge w_4)=0$, and
$$[w'_3,w_4]=[w_1+cw_2-\epsilon F(w'_3)v,w_2+dw_1]=-d\epsilon F(w'_3)[v,w_1]=0,$$
i.e.,  $[v,w_1]=0$. Thus $\mathbf{P}\subset\mathfrak{t}_1$.
Now it follows from the assumption  $\mathrm{rk}\mathfrak{g}=2$ that $\mathbf{P}=\mathfrak{t}_1$. This completes  the proof of
the lemma.
\end{proof}

Now we fix the sufficiently small $\epsilon>0$.
The left invariant metric $\tilde{F}_{1,\epsilon}$ on $G$ is real analytic. So for any
$2$-dimensional subspace $\mathbf{P}\neq\mathfrak{t}_1$, the inequality
$K^{\tilde{F}_{1,\epsilon}}(e,y,\mathbf{P})>0$ is satisfied for all nonzero $y$ in
a dense open subset of $\mathbf{P}$.

Let $\mathfrak{t}_2\neq\mathfrak{t}_1$ be another Cartan subalgebra of
$\mathfrak{g}$, and $v_2\in\mathfrak{t}_2$ a nonzero generic  vector. Then for any sufficiently small $\epsilon>0$, We can similarly define the left invariant metric $\tilde{F}_{2,\epsilon}$ from $F$ and $v_2$.

Now we can construct the left invariant metric indicated in the theorem.

Let $v$ be a bi-invariant unit vector in $\mathfrak{t}_1$. Then we can find
two disjoint closed subsets $\mathcal{D}_1$ and $\mathcal{D}_2$ in
$\mathcal{S}=\{u\in\mathfrak{g}|\langle u,u\rangle_{\mathrm{bi}}=1\}$ such that there are two small open neighborhood $\mathcal{U}_1$ and $\mathcal{U}_2$ of $v$ in $\mathcal{S}$ satisfying $\mathcal{U}_1\subset \mathcal{D}_2\subset 
\mathcal{S}\backslash\mathcal{D}_1\subset\mathcal{U}_2$.
Also we have two closed cones in $\mathfrak{g}\backslash 0$, defined by
\begin{eqnarray*}
\mathcal{C}_1&=&\{ty|\forall t> 0\mbox{ and }y\in\mathcal{D}_1\},\mbox{ and }\\
\mathcal{C}_2&=&\{ty|\forall t> 0\mbox{ and }y\in\mathcal{D}_2\}.
\end{eqnarray*}


Viewing $\tilde{F}_{1,\epsilon}$ and $\tilde{F}_{2,\epsilon}$  as Minkowski norms on
$\mathfrak{g}$,
 we now glue $\tilde{F}_{1,\epsilon}|_{\mathcal{C}_1}$ and $\tilde{F}_{2,\epsilon}|_{\mathcal{C}_2}$ to  produce a new Minkowski norm. Let $\mu:\mathfrak{g}\backslash 0\rightarrow[0,1]$ be a smooth positively homogeneous function of degree $0$ (i.e. $\mu(\lambda y)=\lambda \mu(y)$ whenever $\lambda>0$ and $y\neq0$) such that $\mu|_{\mathcal{C}_1}\equiv 1$
and $\mu|_{\mathcal{C}_2}\equiv 0$. Define $F_\epsilon=\mu\tilde{F}_{1,\epsilon}+(1-\mu)\tilde{F}_{2,\epsilon}$.
Obviously $F'_\epsilon$ coincides with $\tilde{F}_{1,\epsilon}$ on $\mathcal{C}_1$,
and with $\tilde{F}_{2,\epsilon}$ on $\mathcal{C}_2$.
When $\epsilon>0$ is small enough, $F_\epsilon$ can be sufficiently $C^4$-close to
$F_0=F$, when restricted to a closed neighborhood of $\mathcal{S}$.
It implies that the strong convexity condition (3) in the definition of Minkowski norms is satisfied  for $F_\epsilon$. The other conditions are easy to check. So $F_\epsilon$ is a Minkowski norm. Using the left actions of $G$, we get a family of left invariant Finsler metrics $F_\epsilon$ on $G$, where $\epsilon>0$ is sufficiently small.

Let us check the (FP) Condition for $F_\epsilon$ for a fixed $\epsilon>0$.
We only need to check at $e$.
Given  any $2$-dimensional subspace $\mathbf{P}\subset\mathfrak{g}$ different
from $\mathfrak{t}_1$, we can find a generic vector $w_1\in\mathbf{P}$ such that a neighborhood of $w_1$ is contained in
$\mathcal{C}_1$, i.e.,  $F_\epsilon$ and $\tilde{F}_{1,\epsilon}$ coincides on a neighborhood of $w_1$, so we have
$K^{F_\epsilon}(e,w_1,\mathbf{P})=K^{\tilde{F}_{1,\epsilon}}(e,w_1,\mathbf{P})>0$.
For $\mathbf{P}=\mathfrak{t}_1$, we can find a generic vector $w_1\in\mathbf{P}$ from $\mathcal{U}_1$. Then similarly we can show that
$K^{F_\epsilon}(e,w_1,\mathbf{P})=K^{\tilde{F}_{2,\epsilon}}(e,w_1,\mathbf{P})>0$.
Thus $F_\epsilon$ satisfies the (FP) Condition.

This completes the proof of Theorem \ref{main-thm-2}.

Finally, we remark that Lemma \ref{lemma-007} and the gluing technique can be applied to discuss the case that $\mathrm{rk}G>2$, the above argument might be refined to be more elegant, which avoids using the speciality that $\mathrm{rk}G=2$.
On the other hand, if $\dim\mathfrak{c}(\mathfrak{g})>1$, then Theorem \ref{simple-flag-curvature-formula-thm} implies
that the flag curvature vanishes for any flag contained in the center,
regardless of the choice of the flag pole. So the (FP) Condition is not satisfied
by any left invariant Finsler
metric on $G$.

These observations encourage us to propose the following problem.

\begin{problem} \label{question-3}
Let $G$ be a Lie group such that $\mathrm{Lie}(G) =\mathfrak{g}$ is a compact non-Abelian Lie algebra with  $\dim\mathfrak{c}(\mathfrak{g})\leq 1$. Does $G$ admit a flag-wise positively curved left invariant metric?
\end{problem}

\section{A key technique for the homogeneous positive flag curvature problem}
In this section, we prove Theorem \ref{mainthm-1}. We first recall some result on odd dimensional positively curved homogeneous Finsler spaces.
\subsection{General theories for odd dimensional positively curved homogeneous Finsler spaces}

Let $(G/H,F)$ be an odd dimensional positively curved homogeneous Finsler space such that $\mathfrak{g}=\mathrm{Lie}(G)$ and $\mathfrak{h}=\mathrm{Lie}(H)$ are compact Lie algebras.
We fix a bi-invariant inner product $\langle\cdot,\cdot\rangle_{\mathrm{bi}}$ of $\mathfrak{g}$, and an  orthogonal decomposition
$\mathfrak{g}=\mathfrak{h}+\mathfrak{m}$ with respect to $\langle\cdot,\cdot\rangle_{\mathrm{bi}}$.

The rank inequality in \cite{XDHH06} implies that
$\mathrm{rk}\mathfrak{g}=\mathrm{rk}\mathfrak{h}+1$. Fix a fundamental Cartan subalgebra $\mathfrak{t}$ of $\mathfrak{g}$, i.e.,
$\mathfrak{t}\cap\mathfrak{h}$ is a Cartan subalgebra of
$\mathfrak{t}$.
Then we have $\mathfrak{t}=\mathfrak{t}\cap\mathfrak{h}+
\mathfrak{t}\cap\mathfrak{m}$ and
$\dim\mathfrak{t}\cap\mathfrak{m}=1$.
We denote $T$ and $T_H$ the maximal tori in $G$ and $H$
generated by $\mathfrak{t}$ and $\mathfrak{t}\cap\mathfrak{h}$
respectively.
In later discussions, roots and root
planes of $\mathfrak{g}$ and $\mathfrak{h}$ are meant to be  with respect to
$\mathfrak{t}$ and $\mathfrak{t}\cap\mathfrak{h}$,  respectively. By the restriction of $\langle\cdot,\cdot\rangle_{\mathrm{bi}}$ to
$\mathfrak{t}$ and $\mathfrak{t}\cap\mathfrak{h}$,
roots of $\mathfrak{g}$ and $\mathfrak{h}$ will also be viewed as vectors in $\mathfrak{t}$ and $\mathfrak{t}\cap\mathfrak{h}$,  respectively.

The most fundamental decomposition for $\mathfrak{g}$ is
$$\mathfrak{g}=\mathfrak{t}+\sum_{\alpha\in\Delta_\mathfrak{g}}
\mathfrak{g}_{\pm\alpha},$$
where $\Delta_\mathfrak{g}\subset\mathfrak{t}$ is the root system of $\mathfrak{g}$ and
$\mathfrak{g}_{\pm\alpha}$ is the root plane of $\mathfrak{g}$ for the roots
$\pm\alpha$.

Another decomposition for $\mathfrak{g}$ is
\begin{equation}\label{decomposition-2}
\mathfrak{g}=\sum_{\alpha'\in\mathfrak{t}\cap\mathfrak{h}}
\hat{\mathfrak{g}}_{\pm\alpha'},
\end{equation}
where $\hat{\mathfrak{g}}_{\pm\alpha'}=
\sum_{\mathrm{pr}_{\mathfrak{h}}(\alpha)=\alpha'}\mathfrak{g}_{\pm\alpha}$
when $\alpha'\neq 0$ and $\hat{\mathfrak{g}}_{0}=
\mathfrak{c}_\mathfrak{g}(\mathfrak{t}\cap\mathfrak{h})$. Notice that
$\hat{\mathfrak{g}}_{0}=\mathfrak{t}+\mathfrak{g}_{\pm\alpha}$ if
there are roots $\pm\alpha\in\mathfrak{t}\cap\mathfrak{m}$, and otherwise
$\hat{\mathfrak{g}}_{0}=\mathfrak{t}$. The following lemma indicates that this decomposition is compatible
with the orthogonal decomposition $\mathfrak{g}=\mathfrak{h}+\mathfrak{m}$,
and the decomposition
$$\mathfrak{h}=\mathfrak{t}\cap\mathfrak{h}+\sum_{\alpha'\in\Delta_\mathfrak{h}}
\mathfrak{h}_{\pm\alpha'},$$
where $\Delta_\mathfrak{h}\subset\mathfrak{t}\cap\mathfrak{h}$ is the root system of $\mathfrak{h}$ and
$\mathfrak{h}_{\pm\alpha'}$ is the root plane of $\mathfrak{h}$ for the roots
$\pm\alpha'$.

\begin{lemma} \label{Lie-alg-decomposition-lemma-1}
For any $\alpha'\in\mathfrak{t}\cap\mathfrak{h}$, we have the following.
\begin{description}
\item{\rm (1)} $\hat{\mathfrak{g}}_{\pm\alpha'}=
    \hat{\mathfrak{g}}_{\pm\alpha'}\cap\mathfrak{h}
+\hat{\mathfrak{g}}_{\pm\alpha'}\cap\mathfrak{m}$.
\item{\rm (2)}
If $\hat{\mathfrak{g}}_{\pm\alpha'}\cap\mathfrak{h}\neq 0$, then we have $\pm\alpha'\in\Delta_\mathfrak{h}$ and
$\mathfrak{h}_{\pm\alpha'}=\hat{\mathfrak{g}}_{\pm\alpha'}\cap\mathfrak{h}$.
\item{\rm (3)}
Denote $\hat{\mathfrak{m}}_{\pm\alpha'}=\hat{\mathfrak{g}}_{\pm\alpha'}\cap\mathfrak{m}$.
Then
$\hat{\mathfrak{m}}_{\pm\alpha'}=\hat{\mathfrak{g}}_{\pm\alpha'}$ if and only if
$\alpha'\notin\Delta_\mathfrak{h}$. In particular, $\hat{\mathfrak{m}}_0=\hat{\mathfrak{g}}_0\cap\mathfrak{m}$ is either one dimensional or three dimensional.
\end{description}
\end{lemma}
By this lemma,  the decomposition (\ref{decomposition-2}) implies that
\begin{equation}\label{decomposition-2-for-m}
\mathfrak{m}=\sum_{\alpha'\in\mathfrak{t}\cap\mathfrak{h}}
\hat{\mathfrak{m}}_{\pm\alpha'}.
\end{equation}
The proof of Lemma \ref{Lie-alg-decomposition-lemma-1} is easy and will be omitted.

Let $\alpha'$ be a nonzero vector in $\mathfrak{t}\cap\mathfrak{h}$.
Denote by $\mathfrak{t}'$ the  orthogonal complement of $\alpha'$ in $\mathfrak{t}\cap\mathfrak{h}$
with respect to $\langle\cdot,\cdot\rangle_{\mathrm{bi}}$, and $\mathrm{pr}_{\mathfrak{t}'}$ the orthogonal projection to $\mathfrak{t}'$. Then we have another decomposition
\begin{equation}\label{decomposition-3-for-m}
\mathfrak{m}=\sum_{\gamma''\in\mathfrak{t}'}
\hat{\hat{\mathfrak{m}}}_{\pm\gamma''},
\end{equation}
where $\hat{\hat{\mathfrak{m}}}_{\pm\gamma''}=\sum_{
\gamma'\in\mathfrak{t}\cap\mathfrak{h},
\mathrm{pr}_{\mathfrak{t}'}(\gamma')=\gamma''}
\hat{\mathfrak{m}}_{\pm\gamma'}$.

The importance of (\ref{decomposition-2-for-m}) and (\ref{decomposition-3-for-m}) lies in the fact that they are orthogonal with
respect to certain inner product $g_y^F$ defined by the homogeneous Finsler metric $F$. We will use the following orthogonality lemma.

\begin{lemma}\label{orthogonality-lemma}
Keep all above assumptions and notations. Then we have the following.
\begin{description}
\item{\rm (1)}
Let $\mathfrak{t}'\subset\mathfrak{t}\cap\mathfrak{h}$ be the Lie algebra for a codimension one sub-torus $T'$ in $T_H$, which is the bi-invariant orthogonal complement
of a nonzero vector $\alpha'\in\mathfrak{t}\cap\mathfrak{h}$.
Then the decomposition (\ref{decomposition-3-for-m})
is $g^F_u$-orthogonal for any nonzero $u\in\hat{\hat{\mathfrak{m}}}_0$.
\item{\rm (2)}
For any nonzero vector
$u\in\hat{\mathfrak{m}}_{\pm\alpha'}$ with $\alpha'\neq 0$, we have
$\langle u,[u,\mathfrak{t}\cap\mathfrak{h}]\rangle_u^F=0$.
\end{description}
\end{lemma}

\begin{proof}
The first statement is very similar to Lemma 3.7 in \cite{XD15}, and so does its proof.

The $\mathrm{Ad}(T')$-actions preserve both $F$ and $\hat{\hat{\mathfrak{m}}}_0$. So for any nonzero vector $u\in\hat{\hat{\mathfrak{m}}}_0$, the inner product $g_u^F$ is $\mathrm{Ad}(T')$-invariant, and the summands in the decomposition (\ref{decomposition-3-for-m}) correspond to different irreducible representations of $T'$. This proves  the $g_u^F$-orthogonality
of (\ref{decomposition-3-for-m}).

(2) By the $\mathrm{Ad}(T_H)$-invariance of $F$, the restriction of
$F$ to the 2-dimensional subspace
$\mathbb{R}u+[\mathfrak{t}\cap\mathfrak{h},u]$ coincides with the bi-invariant
inner product up to a scalar change. Hence the assertion (2) follows immediately.
\end{proof}

\subsection{Proof of Theorem \ref{mainthm-1}}

The main technique in the proof of Theorem \ref{mainthm-1}
can be summarized as the following lemma, which may be applied to more general situations (for example, see Lemma \ref{last-lemma} in Subsection \ref{subsection-2}).

\begin{lemma}\label{key-tech-lemma-1}
Let $(G/H,F)$ be a homogeneous Finsler space with compact Lie algebra $\mathfrak{g}=\mathrm{Lie}(G)$, and compact subalgebra
$\mathfrak{h}=\mathrm{Lie}(H)$ with $\mathrm{rk}\mathfrak{h}=\mathrm{rk}\mathfrak{g}-1$. Fix a bi-invariant
inner product $\langle\cdot,\cdot\rangle_{\mathrm{bi}}$ on $\mathfrak{g}$,
an orthogonal decomposition $\mathfrak{g}=\mathfrak{h}+\mathfrak{m}$
accordingly,
a fundamental Cartan subalgebra $\mathfrak{t}$, and  a bi-invariant
unit vector $u_0\in\mathfrak{t}\cap\mathfrak{m}$.
Assume that $\alpha$ and $\beta$ are roots of $\mathfrak{g}$ satisfying the following conditions:
\begin{description}
\item{\rm (1)} $[\mathfrak{g}_{\pm\alpha},\mathfrak{g}_{\pm\beta}]=0$.
\item{\rm (2)} neither $\alpha$ nor $\beta$ is contained in $\mathfrak{t}\cap\mathfrak{h}$ or $\mathfrak{t}\cap\mathfrak{m}$.
\item{\rm (3)} There exists an $\mathrm{Ad}(T_H)$-invariant decomposition $\mathfrak{m}=\mathfrak{m}'+\mathfrak{m}''+\mathfrak{m}'''$
    such that $\mathfrak{m}'=
    \mathfrak{t}\cap\mathfrak{m}+\mathfrak{g}_{\pm\alpha}$,
    $\mathfrak{m}''=\mathfrak{g}_{\pm\beta}$, $[\mathfrak{m}'+\mathfrak{m}'',\mathfrak{m}''']_\mathfrak{m}\subset
    \mathfrak{m}'''$, and this decomposition is $g_u^F$-orthogonal for any nonzero
    $u\in\mathfrak{m}'\backslash \mathfrak{t}\cap\mathfrak{m}$.
\item{\rm (4)} For any bi-invariant unit vectors $u\in\mathfrak{g}_{\pm\alpha}$ and $v\in\mathfrak{g}_{\pm\beta}$, the values of
$\langle v,v\rangle_{u+tu_0}^F$ only depend on $t$.
\end{description}
Then $(G/H,F)$ can not be positively curved.
\end{lemma}

\begin{proof}
We fix a bi-invariant unit vector $u_0\in\mathfrak{t}\cap\mathfrak{m}$. Let $\{u,\bar{u}\}$
be a bi-invariant orthonormal basis for $\mathfrak{g}_{\pm\alpha}$. 

Denote $u(t)=u+tu_0
\in\mathfrak{m}'\backslash\mathfrak{t}\cap\mathfrak{m}$. 
By the assumption (2) of the lemma, $\alpha\notin\mathfrak{t}\cap\mathfrak{m}$, we can use (2) of Lemma \ref{orthogonality-lemma} to get
\begin{equation}\label{10000}
\langle u(t),[u(t),\mathfrak{t}\cap\mathfrak{h}]
\rangle_{u(t)}^F=
\langle u(t),\mathbb{R}\bar{u}\rangle_{u(t)}^F=0,
\end{equation}
i.e. $\langle u(t),\bar{u}\rangle_{u(t)}^F=0$.

By the assumption (4) of the lemma, i.e. $\langle v,v\rangle_{u(t)}^F$ does not depend on the choice
of bi-invariant unit vector $v$, we have
$\langle v,\bar{v}\rangle_{u(t)}^F=0$ when $v,\bar{v}\in\mathfrak{g}_{\pm\beta}$ are bi-invariant orthogonal. So we have
\begin{equation}\label{10001}
\langle v,[v,\mathfrak{t}]\rangle_{u(t)}^F=0.
\end{equation}
for any $v\in\mathfrak{g}_{\pm\beta}$. Now we fix
a bi-invariant orthonormal basis $\{v,\bar{v}\}$ of
$\mathfrak{g}_{\pm\beta}$.

We will apply the homogeneous flag curvature formula
(\ref{6000}) to the flag curvature $K^F(o,u(t),u(t)\wedge v)$.
In the following, we calculate the components in (\ref{6000})
one by one.

\medskip
(1) The spray vector field $\eta(u(t))$.

By its definition, we have
\begin{eqnarray*}
& &\langle\eta(u(t)),\mathfrak{t}\cap\mathfrak{m}
+\mathbb{R}u(t)+\mathfrak{m}''+\mathfrak{m}'''\rangle_{u(t)}^F\\
&=&\langle u(t),[\mathfrak{t}\cap\mathfrak{m}+
\mathfrak{g}_{\pm\beta}+\mathfrak{m}''',u(t)]_\mathfrak{m}
\rangle_{u(t)}^F\\
&=&\langle u(t),\mathbb{R}\bar{u}+
\mathfrak{g}_{\pm\beta}+\mathfrak{m}'''\rangle_{u(t)}^F=0.
\end{eqnarray*}
Notice that we have used (\ref{10000}), and the assumptions (1) and (3) here. Using (\ref{10000}) and the assumption (3) again,
we get
\begin{equation}\label{10010}
\eta(u(t))=c_1(t)\bar{u},
\end{equation}
where the real function $c_1(t)$ depends smoothly on $t$.

\medskip
(2) The Cartan tensor $C^F_{u(t)}(w_1,w_2,\bar{u})$.

If $w_1,w_2\in\mathfrak{m}''=\mathfrak{g}_{\pm\beta}$, then
there exists a vector $h\in\mathfrak{t}\cap\mathfrak{h}$ such
that $\langle h,\alpha\rangle\neq 0$. By Theorem 1.3 in \cite{DH04}, we have
\begin{equation}\label{10002}
\langle[h,w_1],w_2\rangle_{u(t)}^F+
\langle[h,w_2],w_1\rangle_{u(t)}^F+
2C^F_{u(t)}(w_1,w_2,[h,u(t)])=0.
\end{equation}
By (\ref{10001}), 
$\langle[h,w_1],w_2\rangle_{u(t)}^F=
\langle[h,w_2],w_1\rangle_{u(t)}^F=0$.  
Since $\langle h,\alpha\rangle_{\mathrm{bi}}\neq 0$, 
$[h,u(t)]$ is a nonzero multiple of $\bar{u}$, thus
we get $C^F_{u(t)}(w_1,w_2,\bar{u})=0$.

If $w_1\in\mathfrak{m}''$ and $w_2\in\mathfrak{m}'+\mathfrak{m}'''$,
we can similarly prove $C^F_{u(t)}(w_1,w_2,\bar{u})=0$,
using (\ref{10001}), and the assumption (3) of the lemma.

To summarize, we have
\begin{equation}\label{10003}
C^F_{u(t)}(\mathfrak{m}'',\cdot,\eta(u(t)))\equiv 0.
\end{equation}

\medskip
(3) The connection operator $N(u(t),v)$.

Using (\ref{10003}) and the assumptions (1) and (3), it can be
checked directly from the definition of $N(\cdot,\cdot)$ in (\ref{6000}) that,
for any $w_1\in\mathfrak{m}''$ and $w_2\in\mathfrak{m}'+\mathfrak{m}'''$,
\begin{equation}\label{10004}
\langle N(u(t),w_1),w_2\rangle_{u(t)}^F=0.
\end{equation}

On the other hand, by (\ref{10001}) and
(\ref{10003}), we can get
\begin{eqnarray}\label{10005}
& &\langle N(u(t),v),v\rangle_{u(t)}^F\nonumber\\
&=&\langle[v,u(t)]_\mathfrak{m},v\rangle_{u(t)}^F
+C^F_{u(t)}(v,v,\eta(u(t)))\nonumber\\
&=&\langle [v,u(t)]_\mathfrak{m},v\rangle_{u(t)}^F
= t\langle[v,u_0],v\rangle_{u(t)}^F = 0.
\end{eqnarray}

Summarizing (\ref{10004}) and (\ref{10005}),
we have
\begin{equation}\label{10006}
N(u(t),v)=c_2(t)\bar{v}.
\end{equation}

To determine the functuion $c_2(t)$, we continue with the calculation
\begin{eqnarray}
& &2\langle N(u(t),v),\bar{v}\rangle_{u(t)}^F\nonumber\\
&=&\langle[\bar{v},v]_\mathfrak{m},u(t)\rangle_{u(t)}^F
+\langle[v,u(t)]_\mathfrak{m},\bar{v}\rangle_{u(t)}^F
+\langle[\bar{v},u(t)],v\rangle_{u(t)}^F\nonumber\\
&=&2ct\langle u_0,u_0\rangle_{u(t)}^F+
2c\langle u_0,u\rangle_{u(t)}^F
+t(\langle [v,u_0],\bar{v}\rangle_{u(t)}^F
+\langle[\bar{v},u_0],v\rangle_{u(t)}^F)\nonumber\\
&=&2ct\langle u_0,u_0\rangle_{u(t)}^F+
2c\langle u_0,u\rangle_{u(t)}^F,
\end{eqnarray}
where $c$ is the constant determined by $[\bar{v},v]=2cu_0$, and in the last step, we have used the fact $\langle v,v\rangle_{u(t)}^F=\langle
\bar{v},\bar{v}\rangle_{u(t)}^F$ from the assumption (4) of the lemma. Notice that by the assumption (2), $\beta\notin\mathfrak{h}$, $c$ is nonzero.
So $c_2(t)$ is a nonzero function determined by the following equality,
\begin{eqnarray}\label{10007}
c_2(t)\langle v,v\rangle_{u(t)}^F &=&
c_2(t)\langle \bar{v},\bar{v}\rangle_{u(t)}^F=\langle N(u(t),v),\bar{v}\rangle_{u(t)}^F\nonumber\\
&=&c\langle u_0,u\rangle_{u(t)}^F+ct\langle u_0,u_0\rangle_{u(t)}^F.
\end{eqnarray}

Similar calculation also shows
\begin{equation}\label{10008}
N(u(t),\bar{v})=-c_2(t)v,
\end{equation}
where $c_2(t)$ is the same function indicated in (\ref{10007}).

\medskip
(4) The derivative $D_{\eta(u(t))}N(u(t),v)$.

Denote $u(s,t)=u(s,0)+tu_0$, where $u(s,0)$ is a smooth
family of bi-invariant unit vectors in $\mathfrak{g}_{\pm\alpha}$ with $u(0,0)=u$ and $\frac{\partial}{\partial s}u(0,0)=\bar{u}$. Above calculation for $N(u(t),v)$ can
also be applied to show
$N(u(s,t),v)=c_2(s,t)\bar{v}$
 where $c_2(s,t)$ is determined
by
\begin{equation}\label{10009}
c_2(s,t)\langle v,v\rangle_{u(s,t)}^F
=c\langle u_0,u(s,0)\rangle_{u(s,t)}^F+ct\langle u_0,u_0\rangle_{u(s,t)}^F.
\end{equation}
where $c$ is the nonzero constant defined by $[\bar{v},v]=2cu_0$.

By the assumption (4) of the lemma, we have $\langle v,v\rangle_{u(s,t)}^F=\langle v,v\rangle_{u(t)}^F$.
By the assumption (2) of the lemma, $\alpha$ is not contained in $\mathfrak{t}\cap\mathfrak{m}$,
We can find a group element $g\in T_H\subset H$, such that $\mathrm{Ad}(g)u(s,0)=u$. Then the $\mathrm{Ad}(H)$-invariance
of $F$ implies
\begin{eqnarray*}
\langle u_0,u(s,0)\rangle_{u(s,t)}^F
=\langle\mathrm{Ad}(g)u_0,\mathrm{Ad}(g)u(s,0)
\rangle_{\mathrm{Ad}(g)u(s,t)}^F
=\langle u_0,u\rangle_{u(t)}^F,
\end{eqnarray*}
and similarly $\langle u_0,u_0\rangle_{u(s,t)}^F=\langle u_0,u_0\rangle_{u(t)}^F$.

So compare (\ref{10007}) and (\ref{10009}), we see
$c_2(s,t)=c_2(t)$ and thus $N(u(s,t),v)$ are independent of $s$.
This fact, together with (\ref{10010}) implies
\begin{eqnarray}\label{10011}
D_{\eta(u(t))}N(u(t),v)=0.
\end{eqnarray}

\medskip
(5) The nominator $\langle R^F_{u(t)}v,v\rangle_{u(t)}^F$ of the flag curvature $K^F(o,u(t),u(t)\wedge v)$.

Summarizing the components of (\ref{6000}) we have calculated, i.e.
\begin{eqnarray*}
D_{\eta(u(t))}N(u(t),v)&=&0,\\
N(u(t),N(u(t),v))&=&-c_2(t)^2v,\\
{[}u(t),N(u(t),v){]}_\mathfrak{m}&=& N(u(t),{[}u(t),v{]}_\mathfrak{m})
=c'c_2(t)tv,
\end{eqnarray*}
where $c'$ is the nonzero constant determined by $[u_0,\bar{v}]=c'v$ and $[u_0,v]=-c'v$. Moreover, using
the assumption (1) of the lemma, it
is obvious that $[u(t),v]_\mathfrak{h}=0$.

Now input all these information into the homogeneous flag curvature
formula (\ref{6000}), then we get
\begin{equation}\label{10012}
\langle R_{u(t)}v,v\rangle_{u(t)}^F=
c_2(t)^2\langle v,v\rangle_{u(t)}^F.
\end{equation}

\medskip
Finally, it is easily seen that when $t$ goes to $\pm\infty$,
the inner product $g_{u(t)}^F$ converges to $g^F_{\pm u_0}$, so there are positive constants
$0<C_1<C_2$, such that for all $t\in(-\infty,\infty)$, we have
\begin{eqnarray*}
& &C_1<\langle v,v\rangle_{u(t)}^F<C_2,\\
& &C_1<\langle u_0,u_0\rangle_{u(t)}^F<C_2,\mbox{ and }\\
& &|\langle u_0,u\rangle_{u(t)}^F|<C_2.
\end{eqnarray*}
So by (\ref{10007}) and the continuity of $c_2(t)$,
there must exist $t_0\in\mathbb{R}$ with $c_2(t_)0=0$. Then
by (\ref{10012}), we have $K^F(o,u(t_0),u(t_0)\wedge v)=0$,
which ends the proof of the lemma.
\end{proof}

\medskip
\noindent {\bf Proof of Theorem \ref{mainthm-1}.}
Assume conversely that for the odd dimensional positively curved homogeneous Finsler space $(G/H,F)$, with a bi-invariant orthogonal decomposition
$\mathfrak{g}=\mathfrak{h}+\mathfrak{m}$, we can find roots $\alpha$ and
$\beta$ satisfying (1)-(3) in the theorem. Now we
show that their root planes satisfy all the assumptions in
Lemma \ref{key-tech-lemma-1}.

By the assumption (3) of the theorem, none of $\alpha\pm\beta$ is a root
of $\mathfrak{g}$, so
$[\mathfrak{g}_{\pm\alpha},\mathfrak{g}_{\pm\beta}]=0$, which
proves (1) in Lemma \ref{key-tech-lemma-1}.

If $\alpha$ is contained in $\mathfrak{t}\cap\mathfrak{h}$,
we can use the assumption (2) of the theorem and Lemma 3.9
in \cite{XD15} to claim $\alpha$ is a root of $\mathfrak{h}$.
This is a contradiction to the assumption (1) of the theorem.
If $\beta$ is contained in $\mathfrak{t}\cap\mathfrak{h}$,
we can use the assumption (3) of the theorem, and similarly
argue that $\beta$ is a root of $\mathfrak{h}$, which is
a contradiction. Notice that $\alpha$ and $\beta$ are bi-invariant
orthogonal to each other. So neither of them can be contained in
$\mathfrak{t}\cap\mathfrak{m}$, otherwise the other one is
contained in $\mathfrak{t}\cap\mathfrak{h}$ which brings
the contradiction. So (2) of Lemma \ref{key-tech-lemma-1} is
satisfied.

Let $\mathfrak{t}'$ be the bi-invariant orthogonal complement
of $\alpha'=\mathrm{pr}_\mathfrak{h}(\alpha)$ in
$\mathfrak{t}\cap\mathfrak{h}$. Then (1) of Lemma
\ref{orthogonality-lemma} indicates the decomposition
$\mathfrak{m}=\sum_{\gamma''\in\mathfrak{t}'}
\hat{\hat{\mathfrak{m}}}_{\pm\gamma''}$ is a $g_u^F$-orthogonal
decomposition for any nonzero $u\in\hat{\hat{\mathfrak{m}}}_0$.

By the assumption (2) of the theorem,
we have $$\hat{\hat{\mathfrak{m}}}_0=\mathfrak{m}'=
\mathfrak{t}\cap\mathfrak{m}+\mathfrak{g}_{\pm\alpha}.$$
Using the assumption (3) of the theorem, for $\beta''=\mathrm{pr}_{\mathfrak{t}'}(\beta)$, we have
$$\hat{\hat{\mathfrak{m}}}_{\pm\beta''}=\mathfrak{m}''=
\mathfrak{g}_{\pm\beta}.$$
Let $$\mathfrak{m}'''=\sum_{\gamma''\neq 0,\gamma''\neq\pm\beta''} \hat{\hat{\mathfrak{m}}}_{\pm\gamma''}$$
be the sum of all other summands.
Then the decomposition $\mathfrak{m}=\mathfrak{m}'+\mathfrak{m}''+\mathfrak{m}'''$
is $\mathrm{Ad}(T_H)$-invariant and $g^F_u$-orthogonal
for any nonzero vector $u\in\mathfrak{m}'$.
By the assumptions (2) and (3) of the theorem, it can be
easily checked
that
$$[\mathfrak{m}'+\mathfrak{m}'',
\mathfrak{m}''']_\mathfrak{m}\subset\mathfrak{m}'''.$$
So (3) of Lemma \ref{key-tech-lemma-1} is also valid.

To verify (4) of Lemma \ref{key-tech-lemma-1}, we consider any bi-invariant unit vectors $u,u'\in\mathfrak{g}_{\pm\alpha}$
and $v,v'\in\mathfrak{g}_{\pm\beta}$. Notice that
$\alpha'=\mathrm{pr}_\mathfrak{h}(\alpha)$ and $\beta'=\mathrm{pr}_\mathfrak{h}(\beta)$ are linearly independent in $\mathfrak{t}\cap\mathfrak{h}$ (which implies
$\mathrm{rk}\mathfrak{h}\geq 2$), so
we can find group elements $g_1, g_2\in T_H$, such that
\begin{eqnarray*}
\mathrm{Ad}(g_1)u'=u, & & \mathrm{Ad}(g_1)v=v,\\
\mathrm{Ad}(g_2)u'=u',&\mbox{ and }&
\mathrm{Ad}(g_2)v'=v.
\end{eqnarray*}
Then we have
$$\langle v',v'\rangle^F_{u'+tu_0}
=\langle \mathrm{Ad}(g_1g_2)v',\mathrm{Ad}(g_1g_2)v'
\rangle_{\mathrm{Ad}(g_1g_2)(u'+tu_0)}^F=\langle v,v\rangle_{u+tu_0}^F.
$$
So the values of $\langle v,v\rangle_{u+tu_0}$ do not
depend on the choices for the bi-invariant unit vectors $u$ and $v$, which proves (4) of Lemma \ref{key-tech-lemma-1}.

To summarize, as all requirements are satisfied, Lemma \ref{key-tech-lemma-1} indicates the homogeneous Finsler space $(G/H,F)$ is not positively curved, which is a contradiction.

\section{Proof of Theorem \ref{cor-main-thm-1}}
We conclude this paper by applying Theorem \ref{mainthm-1}
to the coset space $M$ in Theorem \ref{main-thm-1}. In
Subsection \ref{subsection-1}, we prove (1) of Theorem \ref{cor-main-thm-1}, corresponding to the case $k=1$, by
a case by case discussion about the coset spaces in (\ref{mainlist-2}), where we use the notations and convention in \cite{XW} for root systems of compact simple Lie algebras.
In Subsection \ref{subsection-2}, we prove (2) and (3)
of Theorem \ref{cor-main-thm-1}, where we use the notations
for the bi-invariant orthogonal decompositions in the proof
of Theorem \ref{main-thm-1}.

\subsection{Proof for (1) of Theorem \ref{cor-main-thm-1}}
\label{subsection-1}

\noindent {\bf Case 1}\quad $G/H=\mathrm{SU}(p+q)/\mathrm{SU}(p)\mathrm{SU}(q)$ with
$p>q=2$ or $p=q>3$. 

The root system of $\mathfrak{g}=\mathfrak{su}(p+q)$
can be isometrically identified with
$$\{e_i-e_j,\forall 1\leq i<j\leq p+q\},$$
and the fundamental Cartan subalgebra $\mathfrak{t}$ can be chosen so that
$$\mathfrak{t}\cap\mathfrak{m}=\mathbb{R}[q(e_1+\cdots+e_p)
-p(e_{p+1}+\cdots+e_{p+q})].$$
Now consider $\alpha=e_1-e_{p+1}$ and $\beta=e_2-e_{p+2}$. Then they satisfies
(1)-(4) in Theorem \ref{mainthm-1}. So $\mathrm{SU}(p+q)/\mathrm{SU}(p)\mathrm{SU}(q)$ does not admit any positively curved homogeneous Finsler metric.

\medskip
\noindent {\bf Case 2}\quad $G/H=\mathrm{Sp}(n)/\mathrm{SU}(n)$ with $n>4$. The root system of
$\mathfrak{g}=\mathfrak{sp}(n)$ can be isometrically identified with
$$\{\pm e_i,\forall 1\leq i\leq n;\pm e_i\pm e_j,
\forall 1\leq i<j\leq n\},$$
and the fundamental Cartan subalgebra $\mathfrak{t}$ can be chosen so that
$$\mathfrak{t}\cap\mathfrak{m}=\mathbb{R}(e_1+\cdots+e_n).$$
Now consider $\alpha=2e_1$ and $\beta=2e_2$. Then it is easily seen that they satisfy the conditions (1)-(4) in
Theorem \ref{mainthm-1}. Hence
$\mathrm{Sp}(n)/\mathrm{SU}(n)$ do not admit any positively curved
homogeneous Finsler metric.

\medskip
\noindent {\bf Case 3}\quad $G/H=\mathrm{SO}(2n)/\mathrm{SU}(n)$ with $n=5$ or $n>6$.
The root system of
$\mathfrak{g}=\mathfrak{so}(2n)$ can be isometrically identified with
$$\{\pm e_i\pm e_j,
\forall 1\leq i<j\leq n\},$$
and the fundamental Cartan subalgebra $\mathfrak{t}$ can be chosen so that
$$\mathfrak{t}\cap\mathfrak{m}=\mathbb{R}(e_1+\cdots+e_n).$$
Now consider $\alpha=e_1+e_2$ and $\beta=e_3+e_4$. It is easy to check that   they satisfy the conditions (1)-(4) in
Theorem \ref{mainthm-1}. Thus
$\mathrm{SO}(2n)/\mathrm{SU}(n)$ do not admit any positively curved
homogeneous Finsler metric.

\medskip
\noindent {\bf Case 4}\quad $G/H=\mathrm{E}_6/\mathrm{SO}(10)$. The root system of
$\mathfrak{g}=\mathfrak{e}_6$ can be isometrically identified with
$$\{\pm e_i\pm e_j, \forall1\leqq i<j\leqq 5;\pm\tfrac12 e_1\pm\cdots\pm\tfrac12 e_5\pm\tfrac{\sqrt{3}}{2}e_6
	\mbox{ with an odd number of +'s}\},$$
and the fundamental Cartan subalgebra $\mathfrak{t}$ can be chosen so that
$\mathfrak{t}\cap\mathfrak{m}=\mathbb{R}e_6$.
Now set 
\begin{eqnarray*}
\alpha&=&-\frac12e_1+\frac12e_2+\frac12e_3+\frac12e_4+\frac12e_5
+\frac{\sqrt{3}}{2}e_6,\mbox{ and}\\
 \beta&=&-\frac12e_1-\frac12e_2-\frac12e_3-\frac12e_4-\frac12e_5
-\frac{\sqrt{3}}{2}e_6.
\end{eqnarray*}
Then $\alpha, \beta $  satisfy the conditions  (1)-(4) in
Theorem \ref{mainthm-1}. Thus
$\mathrm{E}_6/\mathrm{SO}(10)$ do not admit any positively curved
homogeneous Finsler metric.

\medskip
\noindent {\bf Case 5}\quad $G/H=\mathrm{E}_7/\mathrm{E}_6$. The root system of
$\mathfrak{g}=\mathfrak{e}_7$ can be isometrically identified with
\begin{eqnarray*}
& &\{\pm e_i\pm e_j,\forall 1\leqq i<j<7;\pm\sqrt{2}e_7;
\tfrac12(\pm e_1\pm\cdots\pm e_6\pm\sqrt{2}e_7)\nonumber\\
& &\mbox{  with an odd number of +'s among the first six coefficients}\},
\end{eqnarray*}
and the fundamental Cartan subalgebra $\mathfrak{t}$ can be chosen so that
$\mathfrak{t}\cap\mathfrak{m}=\mathbb{R}(\sqrt{2}e_6+e_7)$. Now set $\alpha=e_5+e_6$ and $\beta=e_5-e_6$.
Then $\alpha,\beta$  satisfy the conditions  (1)-(4) in
Theorem \ref{mainthm-1}. Thus
$\mathrm{E}_7/\mathrm{E}_6$ do not admit any positively curved
homogeneous Finsler metric.

To summarize, the proof for (1) of Theorem \ref{cor-main-thm-1} is completed.

\subsection{Proof for (2) and (3) of Theorem \ref{cor-main-thm-1}}
\label{subsection-2}

We first recall the notations for the bi-invariant orthogonal
decompositions for each $N_i$ and $M$.

We fix the bi-invariant inner product for each $\mathfrak{g}_i$, and
assume for each $i$, $\mathfrak{g}_i=\mathfrak{k}_i+\mathfrak{p}_i$ with
$\mathfrak{k}_i=\mathfrak{h}_i+\mathbb{R}v_i$
is the bi-invariant orthogonal decomposition for the irreducible factor $N_i=G_i/K_i=G_i/(S^1\cdot H_i)$.
Fix a Cartan subalgebra $\mathfrak{t}$ of $\mathfrak{g}=\bigoplus_{i=1}^k\mathfrak{g}_i$
from $\mathfrak{k}=\bigoplus_{i=1}^k\mathfrak{k}_i$. The bi-invariant orthogonal
decomposition for $M=G/H'=G/(T^{k-1}\cdot H)$ is then
$\mathfrak{g}=\mathfrak{h}'+\mathfrak{m}$, where
$\mathfrak{m}=\mathfrak{m}_0+\sum_{i=1}^k\mathfrak{p}_i$,
$\mathfrak{m}_0=\mathfrak{t}\cap\mathfrak{m}=\mathbb{R}v$,
$v=\sum_{i=1}^k c_i v_i$ with $c_i\neq 0$ for each $i$, and
$\mathfrak{h}'=\bigoplus_{i=1}^k\mathfrak{h}_i
+\mathrm{Lie}(T^{k-1})$ is the bi-invariant orthogonal complement of $\mathfrak{m}_0$ in $\mathfrak{k}$. Notice
that the choice of $M$ is totally determined by the choices
for $N_i$'s and the choice for $\mathfrak{m}_0$.

\medskip
The proof for (3) of Theorem \ref{cor-main-thm-1}, i.e. the case when $k>3$, is simple.
We can find two root planes
$\mathfrak{g}_{\pm\alpha}$ and $\mathfrak{g}_{\pm\beta}$
contained in $\mathfrak{p}_1$ and $\mathfrak{p}_2$ respectively. Then direct checking shows $\alpha$ and $\beta$
are the root pair indicated in Theorem \ref{mainthm-1}. So
in this case, $M$ does not admit positively curved homogeneous Finsler metrics.

\medskip
Finally we prove (2) of Theorem \ref{cor-main-thm-1}.

When $k=3$, there are two cases.

\medskip
\noindent {\bf Case 1}\quad
$\mathfrak{g}_1\neq A_1$.

We can find a root
plane $\mathfrak{g}_{\pm\beta}$ from $\mathfrak{p}_1$ such that
$\beta\notin\mathbb{R}v_1$. Let $\mathfrak{g}_{\pm\alpha}$ be a
root plane in $\mathfrak{p}_2$. Then $\alpha$ and $\beta$ are the root pair indicated in Theorem \ref{mainthm-1}.
So
in this case, $M$ does not admit positively curved homogeneous Finsler metrics.

\medskip
\noindent {\bf Case 2}\quad $\mathfrak{g}_1=\mathfrak{g}_2=\mathfrak{g}_3=A_1$.

We can choose the bi-invariant inner product on each factor,
such that they have the same scale. We can also choose
$v_i$ to be a root for $\mathfrak{g}_i$ respectively. When $v=c_1v_1+c_2v_2+c_3v_3$
satisfies that the absolute values of $c_1$, $c_2$ and $c_3$
are not all the same, then the corresponding $M$ does not
admit positively curved homogeneous Finsler metrics. It can
be argued as following. Assume $|c_2|\neq|c_3|$ for example. We
can take the root $\alpha=v_1$ and $\beta=v_2$. Obviously their root planes are contained in the $\mathfrak{p}_i$-factors. They are the root pair
indicated in Theorem \ref{mainthm-1}, which provides the obstacle to the
positively curved homogeneous Finsler metrics on $M$.

When $k=2$, there are three cases.

\medskip
\noindent {\bf Case 3}\quad $\mathfrak{g}_1\neq A_1$ and $\mathfrak{g}_2\neq A_1$.

We can find root planes $\mathfrak{g}_{\pm\alpha}$ and $\mathfrak{g}_{\pm\beta}$ from $\mathfrak{p}_1$ and $\mathfrak{p}_2$ respectively, such that $\alpha\notin\mathbb{R}v_1$ and $\beta\notin\mathbb{R}v_2$.
Then they are the root pair indicated in Theorem \ref{mainthm-1}. So
in this case, $M$ does not admit positively curved homogeneous Finsler metrics.

\medskip
\noindent {\bf Case 4}\quad $\mathfrak{g}_1\neq A_1$ and $\mathfrak{g}_2= A_1$. Let $\mathfrak{g}_{\pm\alpha}$ be
a root plane in $\mathfrak{p}_1$ such that $\alpha\notin\mathbb{R}v_1$. Let $\beta$ be a root for $\mathfrak{g}_2$. The vector $v$ in $\mathfrak{m}_0$ can be
chosen as $v=cv_1+\beta$. Because $v_1$ and $\alpha$ are linearly independent, and a root system is a finite set, only for finite values of $c$, $cv_1+\mathbb{R}\alpha$ has non-empty
intersection with the root system of $\mathfrak{g}_1$. For any
other value of $c$, $\alpha$ and $\beta$ are the root pair
indicated in Theorem \ref{mainthm-1}, so the corresponding $M$
does not admit positively curved homogeneous Finsler metrics.

\medskip
\noindent {\bf Case 5}\quad
$\mathfrak{g}_1=\mathfrak{g}_2=A_1$.
We can choose the bi-invariant inner product on each factor,
such that they have the same scale. We can also choose
$v_i$ to be a root for $\mathfrak{g}_i$ respectively. The following lemma indicates
only finitely many $M$ may admit positively curved homogeneous
Finsler metrics.

\begin{lemma}\label{last-lemma}
When $c_1\neq\pm c_2$, $c_1\neq \pm\frac12 c_2$,
and $c_1\neq \pm 2c_2$, then the corresponding
$M=G/H'=\mathrm{SU}(2)\times\mathrm{SU}(2)/S^1$ does not
admit positively curved homogeneous Finsler metrics.
\end{lemma}
\begin{proof}
We have noticed that Theorem \ref{mainthm-1} can
be effectively used only when
$\dim\mathfrak{t}\cap\mathfrak{h}'>1$, which is not
satisfied in the case discussed in this lemma. But
we can still use Lemma \ref{key-tech-lemma-1}.

Without loss of generalities, we assume $|c_1|<|c_2|$. We fix a fundamental Cartan subalgebra $\mathfrak{t}$, choose a bi-invariant unit vector $u_0\in\mathfrak{t}\cap\mathfrak{m}$, and denote $\mathfrak{g}_{\pm\alpha}$ and $\mathfrak{g}_{\pm\beta}$
the root planes contained in the first and second $A_1$-factor respectively. Then we can find a group element
$g\in H'$, such that $\mathrm{Ad}(g)$ acts on $\mathfrak{m}'=\mathfrak{t}\cap\mathfrak{m}+
\mathfrak{g}_{\pm\alpha}$ as the identity map, but
rotates $\mathfrak{m}''=\mathfrak{g}_{\pm\beta}$ with
an angle $\theta\notin\mathbb{Z}\pi$. So for any bi-invariant unit vectors
$u,u'\in\mathfrak{g}_{\pm\alpha}$ and $v\in\mathfrak{m}''$, we have
\begin{eqnarray*}
\langle u',v\rangle_{u+tu_0}^F&=&\langle \mathrm{Ad}(g)u',\mathrm{Ad}(g)v\rangle_{\mathrm{Ad}(g)(u+tu_0)}^F
=\langle u',\mathrm{Ad}(g)v\rangle_{u+tu_0}^F\\
&=&\cdots=\langle u',\mathrm{Ad}(g)^k v\rangle_{u+tu_0}^F=\cdots,
\end{eqnarray*}
i.e.
$$\langle u',v\rangle_{u+tu_0}^F=\lim_{k\rightarrow\infty}
\langle u',\frac1k\sum_{i=1}^k\mathrm{Ad}(g)^i v\rangle_{u+tu_0}^F=0,$$
which proves $\langle\mathfrak{m}',\mathfrak{m}''\rangle_{u+tu_0}^F=0$.

Similar argument also provides
\begin{eqnarray*}
\langle v,v\rangle_{u+tu_0}^F=
\langle\mathrm{Ad}(g)v,\mathrm{Ad}(g)v\rangle_{u+tu_0}^F
=\cdots=\langle\mathrm{Ad}(g)^k v,\mathrm{Ad}(g)^k v\rangle_{u+tu_0}^F=\cdots.
\end{eqnarray*}
Because the angle $\mathrm{Ad}(g)$ rotates $\mathfrak{m}''=\mathfrak{g}_{\pm\beta}$ with is not
integer multiples of $\pi$, $\mathrm{Ad}(g)^k v$ are bi-invariant unit vectors in $\mathfrak{g}_{\pm\beta}$ which
 are not contained in a line. So
$\langle w,w\rangle_{u+tu_0}^F=c(t,u)
\langle w,w\rangle_{\mathrm{bi}}$ for all $w\in\mathfrak{m}''$.
Using $\mathrm{Ad}(H')$-actions again, we can change $u$ to any other bi-invariant unit vector, so $c(t,u)$ only depends on $t$.

We have verified (4) and part of (3) in Lemma \ref{key-tech-lemma-1}. As $\mathfrak{m}'''=0$, all other assumptions
in (1)-(3) of Lemma \ref{key-tech-lemma-1} can be easily checked. So $M$ does not admit a positively curved homogeneous
Finsler metric in this case.
\end{proof}

Summarize Case 1-5 in this subsection, we see in each case, there may only exist finite choices of $M$ which admit positively curved homogeneous Finsler metrics. This proves
(2) of Theorem \ref{cor-main-thm-1}.

\end{document}